\documentclass[12pt]{amsart}
\usepackage{hyperref,color}
\usepackage[margin=1.2in]{geometry}

\theoremstyle{plain}
\newtheorem{cor}{Corollary}
\newtheorem{theo}[cor]{Theorem}
\newtheorem{lemma}[cor]{Lemma}
\newtheorem{prop}[cor]{Proposition}

\theoremstyle{definition}

\newtheorem{remark}[cor]{Remark}

\numberwithin{cor}{section}
\numberwithin{equation}{section}

\newcommand{\R}{\mathbb R}
\def\C{\mathcal{C}}
\def\eps{\varepsilon}
\newcommand{\rn}{\rbig^n}

\newcommand{\rbig}{{\mathbb{R}}}
\newcommand{\re}[1]{(\ref{#1})}
\newcommand{\begeqa}{\begin{eqnarray}}
\newcommand{\eneqa}{\end{eqnarray}}
\newcommand{\begeqaet}{\begin{eqnarray*}}
\newcommand{\eneqaet}{\end{eqnarray*}}
\newcommand{\beeq}{\begin{equation}}
\newcommand{\eeq}{\end{equation}}
\newcommand{\beeqs}{\begin{equation*}}
\newcommand{\eeqs}{\end{equation*}}

\newcommand{\lst}{{s_\star}}
\newcommand{\xun}{x_1}
\newcommand{\dom}{\partial \Omega}
\newcommand{\tlam}{T_s}
\newcommand{\tmu}{T_\mu}
\newcommand{\dlam}{D_s}
\newcommand{\wlst}{w_{\lst}}
\newcommand{\tlst}{T_{\lst}}
\newcommand{\slst}{\Sigma_\lst}
\newcommand{\slam}{\Sigma_s}

\newcommand{\wlam}{w_s}
\newcommand{\xlam}{x^s}
\newcommand{\ov}{\overline}

\newcommand{\lsc}{\langle}
\newcommand{\rsc}{\rangle}
\newcommand{\ep}{\varepsilon}

\newcommand{\dist}{\mathrm{dist}}

\begin{document}
\title[Overdetermined problems for fully nonlinear elliptic operators]{Overdetermined problems for fully nonlinear elliptic equations}

\author[L. Silvestre]
{ Luis Silvestre}
\author[B. Sirakov]
{Boyan Sirakov }
\
\date{}
\address{L. Silvestre\hfill\break\indent
University of Chicago\hfill\break\indent
Department of Mathematics\hfill\break\indent
5734 S. University Avenue\hfill\break\indent
Chicago, Illinois 60637\hfill\break\indent
USA}
\email{{\tt  luis@math.uchicago.edu}}

\address{B. Sirakov\hfill\break\indent
Pontificia Universidade Cat\'olica do Rio de Janeiro (PUC-Rio)\hfill\break\indent
Departamento de Matem\'atica\hfill\break\indent
Rua Marques de S\~{a}o Vicente, 225, G\'avea\hfill\break\indent
Rio de Janeiro - RJ, CEP 22451-900,\hfill\break\indent
Brasil
}\email{{\tt bsirakov@mat.puc-rio.br}}

%%%%%%%%0.Abstract

\begin{abstract}
We prove that the existence of a solution to a fully nonlinear elliptic equation in a bounded domain $\Omega$ with an overdetermined boundary condition prescribing both Dirichlet and Neumann constant data forces the domain $\Omega$ to be a ball. This is a generalization of Serrin's classical result from 1971.
\end{abstract}


\maketitle

%%%%%%%%%%%%%%%%%%%%%%%%%%%%%%%%%%%%%%%%%%%%%%%%%%%%%%%%%%%%%%%%%%%%%%%%%%%%%%%%%%%%%%%%%%%%%%%%%%%%%%%%
%%%%%%%%%%%%%1.Introducci\'{o}n
%%%%%%%%%%%%%%%%%%%%%%%%%%%%%%%%%%%%%%%%%%%%%%%%%%%%%%%%%%%%%%%%%%%%%%%%%%%%%%%%%%%%%%%%%%%%%%%%%%%%%%%%
\section{Introduction}
This paper is a contribution to the study of overdetermined boundary-value problems for elliptic PDE, started by the celebrated paper of Serrin \cite{Se}.
We will be interested in  {\it fully nonlinear} equations such as
\beeq\label{princ}
\left\{
\begin{array}{rclcl}
F(D^2u, |D u|) +f(u)&=&0&\mbox { in }& \Omega\\
u&>&0&\mbox { in }& \Omega\\
u&=&0&\mbox{ on }& \partial\Omega,\\
|D u|&=&c_0&\mbox{ on }& \partial\Omega.
\end{array}
\right.
\eeq
Throughout the paper $f$ is a locally Lipschitz continuous function on $\R_+=[0,\infty)$, $c_0\in \R_+$, $F$ is a function on $\S_n\times \R_+$, where  $\S_n$ is the space of symmetric $n\times n$ matrices, $F(0,0)=0$, and $\Omega$ is a bounded domain in $\rn$, $n\ge2$, with $C^{2,\gamma}$-smooth boundary, for some $\gamma>0$. The equation \re{princ} is understood in the viscosity sense.

Serrin's theorem states that if $F(D^2u, |D u|)$ is the Laplacian (or $F$ is replaced by a member of a class of more general quasilinear operators), and a classical solution of \re{princ} exists, then $\Omega$ is a ball.  A very large number of extensions of Serrin's theorem can be found in the literature, and  recent years have seen an explosion of works on overdetermined elliptic problems. It is virtually impossible to give a full bibliography, we refer to \cite{AM, BD, BNST, BH, BK, CS, CH, En, ES, FK, FV1, FV2, FG, FGK, GL, HHP, MS, Pr, Ra, Re1, Re2, Si, We} and the references in these papers, for related symmetry results for various degenerate operators, in various geometries, and a variety of methods of proof. Below we discuss in more detail how our results compare to the previous works which deal with fully nonlinear equations.

Our goal is to extend Serrin's result to general second-order fully nonlinear operators. The first assumption that we make is that the equation is rotationally invariant, that is, $F$ is a function only of the eigenvalues of $D^2u$ and the length of $Du$. This is a natural and necessary assumption if we expect to obtain a radial symmetry result for the solution. In other words, we assume that $F$ is a Hessian operator, that is
\begin{itemize}
\item[(H1)] $F(Q^tMQ,p)=F(M,p)$ for each orthogonal matrix $Q$ and  $M\in \S_n$, $p\in \R_+$.
\end{itemize}

%We are going to assume that  \re{princ} has a viscosity solution $u$ which belongs to the H\"older space $ C^{1,\alpha}(\Omega)\cap C^{2,\alpha}(\Omega_\delta)$ for some fixed $\alpha, \delta>0$, where $\Omega_\delta=\{x\in\Omega\::\: \mbox{dist}(x,\partial\Omega)<\delta\}$, and that  $F$ is uniformly elliptic on $u$, in the following sense:

 In most of the results below we will also assume that $F$ is uniformly elliptic and Lipschitz continuous on $\S_n\times \R_+$, in the following sense
\begin{itemize}
 \item[(H2)] there exist numbers $\Lambda\ge\lambda>0$, $k\ge0$,  such that for any $A,B\in \S_n$, $p,q\in \R_+$,
     \beeq\label{elip}
 \mathcal{M}_{\lambda, \Lambda}^+(A-B) +k|p-q| \ge  F(A,p)-F(B,q)\ge \mathcal{M}_{\lambda, \Lambda}^-(A-B)-k|p-q|.
     \eeq
 \end{itemize}
We denote with $\mathcal{M}_{\lambda, \Lambda}^\pm(M)$  the extremal Pucci operators, and recall that, if $M$ is a symmetric matrix with eigenvalues $\mu_1,\ldots,\mu_n$, then
\beeq
\mathcal{M}_{\lambda, \Lambda}^-(M) = \lambda\sum_{\mu_k>0} \mu_k +\Lambda\sum_{\mu_k<0} \mu_k,\qquad
\mathcal{M}_{\lambda, \Lambda}^+(M) = \Lambda\sum_{\mu_k>0} \mu_k +\lambda\sum_{\mu_k<0} \mu_k.
\eeq

There are some degenerate equations of interest which do not satisfy (H2) a priori, but they do if we restrict the choices of $A, B,p,q$  to the Hessians and gradients $D^2u(x)$, $|D u(x)|$, for particular solutions $u$ and $x \in \Omega$. An alternative hypothesis to (H2) is

\begin{itemize}
 \item[(H$2)^\prime$] The equation \re{princ} has a viscosity solution $u$ which belongs to the H\"older space $ C^{2,\alpha}(\Omega)$ for some fixed $ \alpha>0$,  and there exist numbers $\Lambda\ge\lambda>0$, $k\ge0$, such that \re{elip} holds for all $M,N\in D^2u(\Omega)$, and all $p,q\in  |D u|(\Omega)$.
 \end{itemize}

This condition may be satisfied for particular solutions $u$ of equations that otherwise do not satisfy (H2). This is the case for instance if $-u$ is a strictly convex solution of the Monge-Ampere equation, or more generally, if it is a $k$-convex solution of $$S_k(D^2u) = 1,$$ where $S_k(M)$ is the $k$-th symmetric polynomial evaluated at the eigenvalues of $M$ (see for instance \cite{BNST}). Thus Theorem \ref{theoc1} below applies to such equations too. It is a  common trick for elliptic PDE that one can modify the values of $F$ arbitrarily outside the set of values of $D^2 u(\Omega)\times D u(\Omega)$, to make it satisfy the uniform ellipticity condition (H2).

Observe that we made a regularity assumption on $u$ in  $(H2)^\prime$ but not in (H2). This is because, as we will see in the course of the paper, (H2) actually implies that the viscosity solution of \re{princ} is regular enough for our arguments to apply.

We are going to use Alexandrov-Serrin's original method of moving planes. The main difficulty in applying this method to fully nonlinear equations lies in the application of a crucial ingredient of Serrin's proof, the so-called "corner lemma" (see Lemmas 1 and 2 in \cite{Se}). This lemma is essentially linear. Under some conditions, it is possible to apply it to the linearization of the equation, as was already suggested in Serrin's original paper as a mean to study quasi-linear equations. However, in general the corner lemma fails for nonlinear equations. For instance, if $F$ is a minimal Pucci operator with $\lambda < \Lambda$, it follows from \cite{ASS} (see section \ref{sect-perturb} below) that the equation
$$
\mathcal{M}_{\lambda, \Lambda}^-(D^2w) = 0
$$
 has a  solution which is positive inside and vanishes on the boundary of the intersection of two half-spaces with orthogonal normals, and $w$ is homogeneous of order $2+\alpha$, where $\alpha>0$ if $\lambda< \Lambda$. This of course implies that  $w$ has a zero of order two at the corner points, and the classical corner lemma fails.
\medskip

Our first theorem settles the symmetry question for the general equation \re{princ}, with the only extra hypothesis that the operator $F(M,p)$ be  continuously differentiable in the matrix $M$.
Note that the uniform ellipticity hypothesis (H2) implies that $F$ is Lipschitz continuous but not necessarily $C^1$.

\begin{theo}\label{theoc1} Assume (H1), (H2) or $(H2)^\prime$, and that $F(M,p)$ is  continuously differentiable in  $M$. If there exists  a viscosity solution $u$ of \eqref{princ}, then $\Omega$ is a ball and $u$ is radial.
\end{theo}

To our knowledge, prior to our paper results like Theorem \ref{theoc1} for fully nonlinear operators have appeared only for the particular cases when $F(M)=S_k(M)$ is a symmetric polynomial of the eigenvalues of $M$, see \cite{Re2, BNST, En}, and for equations involving Pucci operators or operators in the form $|Du|^\alpha  \mathcal{M}_{\lambda, \Lambda}^+(D^2u)$, with ellipticity constants sufficiently close to each other, see \cite{BD}. In Section \ref{sect-perturb} we will give an extension of the main theorem in \cite{BD} to operators satisfying (H2), with a short proof which will also play an important role in the proof of Theorem \ref{theoc1}.

%\noindent{\it Remark 1.} We will actually prove a more general statement. Specifically,  it is sufficient that $F$ be $C^1$ only on a neighbourhood in $\S_n$ of  $D^2u(\Gamma)$, where $\Gamma$ is a subset of $\partial\Omega$, to be determined below. We will use this in the proofs of Theorems \ref{theon2} and \ref{theoconv}, below.

%\noindent{\it Remark 2.} It is worth noticing that if $F$ is in addition supposed to be in $C^{1,1}(\S_n)$ then Serrin's original proof applies to \re{princ} -- see Remark 4 below.

Theorem \ref{theoc1} deals with a general fully nonlinear equation under the extra hypothesis that $F$ is $C^1$ in $M$. It is an open problem whether the result holds without this condition, except in the particular cases which we next describe.
\smallskip

A particularly interesting example of a fully nonlinear operator which is not $C^1$ is given by one of the Pucci operators $\mathcal{M}_{\lambda, \Lambda}^-$ or $\mathcal{M}_{\lambda, \Lambda}^+$. We can prove that a fairly general symmetry result still holds for these operators when the space dimension is two, or in higher dimensions if we assume that $\Omega$ is strictly convex. One observation that is crucial for our proof is that the Pucci operators are $C^1$ in the set of non-singular symmetric matrices. Indeed, the discontinuities of the derivative of $\mathcal{M}_{\lambda, \Lambda}^-(M)$ or $\mathcal{M}_{\lambda, \Lambda}^+(M)$ take place only when $M$ has at least one eigenvalue equal to zero.

 We will make the following more general assumption that is satisfied in particular by Pucci's operators, or by extremal operators in the form $\mathcal{M}_{\lambda, \Lambda}^\pm(D^2u)\pm k|D u| $.
\begin{itemize}
 \item[(H3)] $F$ is $C^1$ in $M$ on the set $\{ M \in \S_n\;|\; \mathrm{det}(M)\not=0\}\times \mathbb{R}_+$.
 \end{itemize}

The following theorem contains a  general symmetry statement for  two-dimensional domains.

\begin{theo}\label{theon2} Assume (H1), (H2),  (H3),  and that $\Omega\subset\R^2$ (that is, $n=2$). If in addition $f(0)\ge 0$, then $\Omega$ is a ball and $u$ is radial.
\end{theo}

Finally, we can show that under (H1)-(H3) the only strictly convex domain in $\rn$ for which \re{princ} may have a solution is the ball. By strictly convex, we mean that $\partial \Omega$ is a $C^{2,\alpha}$-surface whose second fundamental form is positive definite (strictly).

\begin{theo}\label{theoconv} Assume (H1), (H2), (H3), and that $\Omega$ is a strictly convex domain. If in addition $f(0)\ge 0$, then $\Omega$ is a ball and $u$ is radial.
\end{theo}

Notice that the assumptions on $f$ in the last two theorems contain as very particular cases the ``torsion" problem $f(u)= 1$ and the ``eigenvalue" problem $f(u)=  \lambda u $.

Theorems \ref{theon2} and \ref{theoconv} are completely new for fully nonlinear operators.

In the end, we comment on the organization of the paper and the proofs of the above theorems. In section \ref{sect-regularity} we collect some boundary regularity results for viscosity solutions to uniformly elliptic equations, most of which are proved in our recent work~\cite{SS}.  In section \ref{sect-movingplanes} we recall the moving planes method as our main strategy to prove that $\Omega$ is a ball. It turns out that, compared to its classical application to the Laplace equation, all steps of this method are easily adaptable to fully nonlinear equations, except for the key step excluding  the so-called \emph{corner (or right-angle) situation}. The purpose of the rest of the paper is then to rule out this corner situation, under suitable assumptions. In section \ref{sect-perturb} we obtain a full symmetry result under (H1) and (H2), provided  the equation is a small nonlinear perturbation of the Laplace equation, that is, the ratio $\Lambda/\lambda$ is sufficiently close to one. The proof in section \ref{sect-perturb} is based on an idea by Birindelli and Demengel in \cite{BD}, and uses the recent results in \cite{ASS}. In section~\ref{sect-proofC1} we give the proof of Theorem \ref{theoc1}, whose main ingredient is an application of the perturbative proof in section \ref{sect-perturb} to a linearized version of \re{princ}, in some sufficiently small neighbourhood of the corner point. Finally, the proofs of Theorems \ref{theon2} and \ref{theoconv} use in addition a measure-theoretic observation, which states that the set of unit normals to the boundary of the domain at points where the Gauss curvature vanishes is negligible on the unit sphere. Theorem \ref{theon2} and Theorem \ref{theoconv} are proved in section \ref{sect-pucci}.

\section{Regularity considerations}
\label{sect-regularity}

In this section we collect some regularity results that apply to the viscosity solutions of the equation \eqref{princ}. If the reader is  willing to assume from the beginning that the solution of \re{princ} is in $C^{2,\alpha}(\overline{\Omega})$, then they may skip this section, after scanning Lemma~\ref{lem1} below and observing it holds with $A= D^2u(x)$ and $b= D u(x)$.

The first regularity result we recall is a consequence of the uniform ellipticity and the theory of Krylov and Safonov. It implies that any solution of \re{princ} belongs to the class $C^{1,\alpha}(\Omega)$, for some $\alpha>0$.

\begin{prop} \label{c1alpha}
Assume that $F(D^2u,Du)$ satisfies (H2),  $\Omega$ is a $C^{2}$-domain, and $g\in C(\overline\Omega)$. Then the solution of
 \beeq\label{equagen}
 F(D^2u, Du) = g(x) \;\mbox{ in }\Omega, \qquad u=0 \;\mbox{ on }\partial\Omega
 \eeq
 is in the class $C^{1,\alpha}(\Omega)$ for some $\alpha > 0$, and
 $$
 \|u\|_{C^{1,\alpha}(\Omega)} \le C\left(\|u\|_{L^\infty(\Omega)}+\|g\|_{L^\infty(\Omega)}\right).
 $$
\end{prop}

This proposition  is well known. It can be found for example as Proposition 2.2 in~\cite{ASS} or Theorem 1.4 in \cite{SS}.

The {\it interior} regularity in the proposition above cannot be improved in general, as the examples in \cite{NV} show. In the special case when the function $F(M,p)$ is assumed to be convex or concave in $M$, it is well known that the solution $u$ belongs to the smoother class $C^{2,\alpha}(\Omega)$.

The next proposition  says that the solutions to fully nonlinear uniformly elliptic equations are $C^{2,\alpha}$ on the boundary, and have a second order Taylor expansion at each boundary point with the corresponding error bounds, under the sole condition (H2).

\begin{prop} \label{prop-expansion-on-the-boundary}
Let $F$ satisfy (H2), $\Omega$ be $C^{2,\gamma}$-smooth, and $g\in C^\gamma(\Omega)$, for some $\gamma>0$. For any solution $u$ of \eqref{equagen} there exist  $A\in C^\alpha(\partial \Omega, \S_n)$ and $b\in C^{1,\alpha}(\partial \Omega, \R^n)$, such that for each $x\in \partial \Omega$ we have $F(A(x),b(x))=0$, and there exists a quadratic polynomial $P_x$ of the form
\[P_x(y) = \frac 12 \lsc A(x)(x-y),x-y \rsc + \lsc b(x),x-y \rsc = \frac 12 a_{ij}(x) (y_i-x_i)(y_j-x_j) + b_i(x) (y_i-x_i)\]
such that for all  $y \in \Omega$, and some $\alpha>0$,
\begin{align}
|u(y) - P_x(y)| &\leq C|x-y|^{2+\alpha},\\
|D u(y) - D P_x(y)| &\leq C|x-y|^{1+\alpha}.
\end{align}
\end{prop}

The fact that classical solutions of the Dirichlet problem for uniformly elliptic equations are $C^{2,\alpha}$ on the boundary of the domain was first proved by Krylov in \cite{krylov1983}.
Extending these results to viscosity solutions turns out to be less trivial than one might expect. This is the subject of our recent work \cite{SS}, in which we also establish the asymptotic expansions in Proposition \ref{prop-expansion-on-the-boundary}. In fact, Proposition \ref{prop-expansion-on-the-boundary} is obtained by applying Theorem 1.2 in \cite{SS} to $u$, and by applying Theorem 1.1 in \cite{SS} to each partial derivative of $u$ in $\Omega$ (these partial derivatives satisfy the inequalities $(S^*)$ in \cite{SS}).

Note that in Proposition \ref{prop-expansion-on-the-boundary} we will always have $b(x) = D u(x)$, by Proposition \ref{c1alpha}. If $u$ is also a $C^2$ function around the boundary $\partial \Omega$ then $A(x) = D^2u(x)$. Because of this, we will abuse notation and write $D^2u = A$,  $(D^2 u)_{ij} = \partial_{ij} u = a_{ij}$. We need to remember that this is not a standard second derivative, but it is understood only in the sense of Proposition \ref{prop-expansion-on-the-boundary} and $D^2 u$ is in general only defined on $\partial \Omega$.

The next proposition, also from \cite{SS}, says that if $F$ is $C^1$ in $M$ then the solution $u$ is actually $C^{2,\alpha}$-smooth in a neighborhood of $\partial \Omega$. Its proof combines Proposition \ref{prop-expansion-on-the-boundary} with a smoothness result for solutions with small oscillations originally due to Ovidiu Savin.

\begin{prop} \label{bdaryC2alpha}
Let the operator  $F(D^2u, Du)$ satisfy (H2), and $\Omega$ be $C^{2,\gamma}$-smooth. If $F(M,p)$ is continuously differentiable in $M$ then any viscosity solution $u$ of \re{princ} is in the class $C^{2,\alpha}(\Omega_\delta)$ for some $\alpha,\delta > 0$, where $\Omega_\delta := \{ x \in \Omega: \dist(x,\partial \Omega) < \delta \}$.
\end{prop}

A local version of this result is also available.

\begin{prop} \label{bdaryC2alpha-local}
Let $F$ satisfy (H2), $\Omega$ be $C^{2,\gamma}$-smooth, and  $u$ be a viscosity solution of \eqref{princ}. Let $x_0 \in \partial \Omega$ and $P$ be a second order polynomial such that $|u(x) - P(x)| = o( |x-x_0|^2)$ for $x\in \Omega$ close to $x_0$. Assume also that $F$ is $C^1$ in $M$ in a neighbourhood of  $(D^2P(x_0),|D P(x_0)|)$. Then $u$ is $C^{2,\alpha}$  in a neighborhood of $x_0$ in $\overline{\Omega}$.
\end{prop}

We finish this section with a result which is not strictly about regularity. In fact it is a general property of functions independent of the equation \eqref{princ}. We will use the following lemma for carrying out the  moving planes method for viscosity solutions which we explain in section \ref{sect-movingplanes}. The precise form of this lemma will also play a crucial role in section \ref{sect-pucci} when we study the overdetermined problem for the Pucci equations.

\begin{lemma}\label{lem1}
Let $\Omega$ be a domain with a $C^2$ boundary and $u\in C^1(\overline \Omega)$ be a function satisfying
\begin{align*}
u &= 0 \text{ on } \partial \Omega, \\
|D u| &= c_0 \text{ on } \partial \Omega.
\end{align*}
Assume that at a point $x \in \partial \Omega$, there exists $A\in \S_n$ and $b\in \R^n$ such that the second order polynomial
$$P(y) =  \frac{1}{2} \lsc A(x-y),x-y \rsc + \lsc b,x-y \rsc$$
satisfies
$$
|u(y)-P(y)| \leq C|x-y|^{2+\alpha}\quad\mbox{and}\quad |D u(y)- D P(y)| \leq C|x-y|^{1+\alpha},\quad\mbox{ for all}\; y \in \Omega.
$$
Then the interior normal vector $\nu(x)$ is an eigenvector of $A=(a_{ij})_{i,j=1}^n$, corresponding to the eigenvalue $a_{nn}$. The other $(n-1)$ eigenvalues of $A$  are $c_0 \kappa_1(x),\ldots,c_0 \kappa_{n-1}(x)$, where $\kappa_1(x),\ldots,\kappa_{n-1}(x)$ are the principal curvatures of $\partial \Omega$ at $x$. The corresponding eigenvectors are the directions of the principal curvatures of $\partial \Omega$ at $x$.
\end{lemma}

\begin{proof}
Without loss of generality let us assume that $x =0$, $\nu = (0,\dots,0,1)$ and the $(n-1)$ principal directions of curvature of $\partial \Omega$ are the first $(n-1)$ coordinate axes.

Note that since $u \in C^1(\overline \Omega)$, we immediately have $b = D u = c_0 \nu = (0,\dots,0,c_0)$, that is,
\[P(y) = \frac 12 \sum_{i,j=1}^{n}a_{ij} y_iy_j + cy_n.\]
 Since $\partial \Omega$ is $C^2$ smooth, there is a $C^2$ function $h$ defined in a neighborhood $V$ of the origin in $\R^{n-1}$ such that $(\tau, h(\tau)) \in \partial \Omega$ for all $\tau \in V$. In particular $u(\tau,h(\tau)) \equiv 0$ and $|D u(\tau,h(\tau))| \equiv c_0$. The eigenvalues of $D^2 h(0)\in \S_{n-1}$ are the principal curvatures of $\partial \Omega$ at $x=0$, and its eigenvectors are the principal directions.

Let $(\tau,0)$ be a vector tangent to $\partial \Omega$ at the origin.
For $\eps$ small, since $h \in C^2$ and $D h(0)=0$ we have that $h(\eps\tau) = \eps^2 \lsc D^2h(0)\tau,\tau \rsc + o(\eps^2)$. Let $z_\eps = (\eps\tau , h(\eps\tau))$. From the definition of $h$, we know that $z_\eps \in \partial \Omega$. We compute, for any $\tau \in \R^{n-1}$,
\begin{align*}
0 &= u(z_\eps) \\
&= P(z_\eps) + O(\eps^{2+\alpha}) \\
&= \frac{\eps^2}2 \sum_{i,j=1}^{n-1}a_{ij} \tau_i \tau_j + c_0 \eps^2\sum_{i,j=1}^{n-1} \frac{\partial_{ij} h(0)} 2 \tau_i \tau_j + o(\eps^2).
\end{align*}
Therefore, $a_{ij} = -c_0 \partial_{ij} h(0)$ for $i,j = 1,\dots,n-1$. This finishes the proof of the second part of the lemma. We are left to prove that $a_{nj} = 0$ for $j=1,\dots,n-1$.

Since $z_\eps \in \partial \Omega$, we know that $D u(z_\eps) = c_0 \nu_\eps$, where $\nu_\eps$ is the inner unit normal vector to $\partial \Omega$ at $z_\eps$. From the assumption, we have that
\[ |D u(z_\eps) - D P(z_\eps)| \leq C |z_\eps|^{1+\alpha} \leq C \eps^{1+\alpha}.\]	
However, since $b = c_0 \nu = (0,\dots,0,c_0)$ we get
\begin{align*}
|D u(z_\eps) - D P(z_\eps)|&\ge
\nu \cdot (D u(z_\eps) - D P(z_\eps))\\
 &= \sum_{i=1}^{n}\nu_i  \left(c_0 (\nu_\eps)_i - \sum_{j=1}^{n}a_{ij} (z_\eps)_j - b_i\right), \\
&= c_0 \nu \cdot (\nu_\eps-\nu) - \sum_{i,j=1}^{n}a_{ij} (z_\eps)_j \nu_i, \\	
&= O(\eps^2) + \eps \sum_{j=1}^{n-1}a_{nj} \tau_j .
\end{align*}

Therefore $\sum_{j=1}^{n-1}a_{nj} \tau_j  = 0$ for any tangential vector $(\tau,0)$. This implies that $a_{nj} = 0$ for $j=1,\dots,n-1$, so $a_{nn}$ is an eigenvalue of $A$, and $\nu$ is the corresponding eigenvector.
\end{proof}

\begin{remark}\label{remafterprop25}
If the function $u$ satisfies the assumptions of both Proposition \ref{prop-expansion-on-the-boundary} and Proposition \ref{lem1} then we obviously have $a_{nn} = u_{\nu\nu}(x)$, the latter derivative being understood in the sense of Proposition \ref{prop-expansion-on-the-boundary}.
\end{remark}

\section{The moving plane method and its corner situation}
\label{sect-movingplanes}

The proofs of the main results of this article are based on the Alexandrov-Serrin's moving planes method (see \cite{Se,BN}), which nowadays is a very standard tool in the theory of elliptic PDE. We recall that the main idea of this method is
to show that for sufficiently many directions $e\in S^{n-1}:=\{x\in\rn\,|\,|x|=1\}$ there exists $s=s(e)\in\rbig$ such that the domain
and the solution are symmetric with respect to the hyperplane
$\tlam=\{x\in\rn\,|\, \lsc x,e\rsc=s\}\:;$  we denote with $\lsc\cdot,\cdot\rsc$ the scalar product in $\rn$.

Fix for instance $e=(1,0,\ldots,0)$ and  set
for any $s\in\rbig$
$$
\begin{array}{rcl}
\tlam&=&\{x\:|\: \xun=s\}\,,\qquad \dlam =\{x\:|\:
\xun>s\}\,,\qquad \slam=\dlam\cap\Omega,\\
\xlam&=&(2s-\xun,x_2,\ldots,x_n)\;-\;\mbox{the reflexion of
}x \mbox{ with
    respect to }\tlam,\\
v_s(x) &=&u(\xlam),\qquad  \wlam(x)=v_s(x)-u(x)\,, \qquad
\mbox{ provided }\,x \in\slam,\\ d_0&=&
\inf\{s\in\rbig\;\:|\:\; \tmu\cap\overline{\Omega}=\emptyset\;
\mbox{ for all }\:\mu>s\}.
\end{array}
$$

It  follows from hypothesis (H1) that the function $v_s$ satisfies the same equation as $u$ in $\slam$, so by (H2) and the Lipschitz continuity of $f$ we get that $\wlam$ is a solution of
\beeq\label{baseineq}
\mathcal{M}_{\lambda, \Lambda}^-(D^2\wlam) - k|D \wlam| -l\wlam\le 0\qquad \mbox{in }\; \slam;
\eeq
from now on $l\ge0$ will denote the Lipschitz constant of $f$ on the interval $[0,\max_\Omega u]$. If $u$ is only a viscosity solution of \re{princ}, see for instance Proposition 2.1 in \cite{dLS} for the derivation of \re{baseineq}. In  \cite{dLS} it is proved that if $\Omega$ is a ball, then any solution of \re{princ} is radial; this result does not require the Neumann hypothesis in \re{princ}.

By using the standard moving planes method, exactly as in the proof of Theorem~1.1 in \cite{Li} or Theorem 1.1 in \cite{dLS} we can show that a hyperplane starting from  $s= d_0$ and moving to the left will move  as long as the reflexion of
$\slam$ with respect to $\tlam$ is contained in $\Omega$, and at least down to
 position $s=\lst$ (the critical position),
where
$$
\lst=\inf\{s\leq d\;\:|\:\;{({\Sigma_\mu})}^\mu\subset\Omega\;
\mbox{ and }
\lsc \nu(x),e\rsc\;<0
\mbox{ for all } \mu>s,\;x\in T_{\mu}\cap \dom\}
$$
(an upper index means reflexion with respect to the
hyperplane with the same index). Recall we denote with $\nu(x)$ the interior normal to $\partial \Omega$ at $x$. In addition, we   have
$\wlam>0$ in $\slam$ for all $s>\lst$, and hence $\wlst\ge 0$ in $\slst$.

Now,  at least one of the following two events occurs:
\begin{itemize}
\item[(i)]\ the reflexion of $\dom\cap\partial\slst$ with respect to $\tlst$ is internally
tangent to $\dom$ at some point $P\in \partial \Omega$ ;
\item[(ii)]\ $\tlst$ is orthogonal to $\dom$ at some point $Q\in \partial\Omega\cap \tlst$.
\end{itemize}

The function $u$, and its reflection $v_s$ are only assumed to be viscosity solutions of \eqref{princ}. In particular, we do not assume that they are $C^2$ functions. However, from Propositions \ref{c1alpha} and \ref{prop-expansion-on-the-boundary}, we know that they are $C^{1,\alpha}$ functions in $\overline \Omega$ with a second order expansion at every point on the boundary $\partial \Omega\cap\partial\slst $. Therefore, the same holds for the function $\wlst$.

In the case (i) occurs we have $\wlst\ge0$ in $\slst$, $\wlst(P^\lst)=\frac{\partial \wlst}{\partial \nu}(P^\lst)=0$. The Hopf lemma (see for instance Proposition \ref{hopf} and the remark following it, below) applied to \re{baseineq} implies $\wlst\equiv0$ in $\slst$, which means $\Omega$ is symmetric in the direction $e$, and we are done.

All the trouble is due to the possibility of orthogonality, that is, of the occurrence of the event (ii). In this place one needs a qualitatively different argument in the fully nonlinear case, compared to the well known results for semilinear and quasilinear equations.

Applying Lemma \ref{lem1}, we get that $u$ and $v_s$ have the same quadratic expansion at $Q$. This is a generalized version, for viscosity solutions,  of Serrin's argument on pages 307-308 of \cite{Se} which applies to functions in $C^2(\overline{\Omega})$. Thus, we get that $D^2 \wlst(Q)=0$. Here $D^2 \wlst$ is understood in the sense of Proposition \ref{prop-expansion-on-the-boundary}. More precisely, we get from that proposition that for some $\alpha>0$,
\beeq\label{equa2}
0 \leq \wlst(x) \leq C |x-Q|^{2+\alpha},
\eeq
for $x \in \Sigma_{s_\star}$.

In the case $\lambda=\Lambda$ Serrin's corner lemma (which can also be seen as the particular case $\lambda=\Lambda$, $\beta=2$, of Proposition~\ref{hopf} below) implies $\wlst\equiv0$ in $\slst$ and the proof is finished. For more general operators a different argument is needed. Our success in showing a symmetry result for each nonlinear equation depends on our finding an argument that implies $w_{s^*} \equiv 0$ in this case. Thus, for the proofs on Theorems \ref{theoc1}, \ref{theon2} and \ref{theoconv}, we just need to address the corner situation.

We end this section with an observation on the value $c_0$ of the Neumann data  in \re{princ}, which will be used in section \ref{sect-pucci}. We can assume that $c_0$ is as small as we like, at the only cost of increasing the Lipschitz constant $l$ in \re{baseineq}. This is because for each $R>0$ the function $u_R = u/R$ is a solution of $F_R(D^2 u, |Du|) + f_R(u) =0$, where $f_R(s) = f(Rs)/R$, and the operator $F_R(M,p) = F(RM, Rp)/R$ satisfies (H1), (H2) and (H3) with the same constants in (H2) as $F$.  It is obvious that the moving planes always reach the same positions for $u_R$ as for $u$.
In addition, if $f(0)\ge 0$ we have $c_0>0$, since $c_0=0$ is excluded by the Hopf lemma, (H2), and the fact that $u$ is a positive solution of $F(D^2u, |Du|) + c(x) u=-f(0)\le 0$ where $c(x)$ is a bounded function (take $c = (f(u)-f(0))/u$ if $u\not=0$ and $c=0$ otherwise).
%From now on we will assume that $\lst=0$, $Q$ is the origin in $\rn$,  and we will write $\bar x$ instead of $x^{\lst}$. We denote with $B_r$  the ball in $\rn$ centered at the origin with radius~$r$.

\section{Symmetry for small perturbations of the Laplacian}
\label{sect-perturb}

In this section, we show that if the two ellipticity constants in (H2) are sufficiently close to each other, then the symmetry result holds. The precise statement is as follows.

\begin{theo}\label{theoperturb} Assume (H1) and (H2), and that $\Omega$ is $C^2$. Assume also that $u$ is a viscosity solution of \eqref{princ} which is $C^{2,\alpha}$ in a neighborhood of $\partial \Omega$, for some $\alpha>0$. There exists a positive number $\epsilon_0$ depending only on $n$ and $\alpha$, such that if $|\Lambda/\lambda-1|<\epsilon_0$ then $\Omega$ is a ball and $u$ is radial.
\end{theo}

This theorem does not require the assumption that $F \in C^1$. The assumption $u \in C^{2,\alpha}(\Omega_\delta)$ is automatically satisfied if $F \in C^1$ (by Proposition \ref{bdaryC2alpha}), or $F$ is concave/convex in the second derivative of $u$ (by the Evans-Krylov theorem).

In the case when $F$ is a Pucci extremal operator, the result in Theorem \ref{theoperturb} is due to Birindelli and Demengel \cite{BD}. We will give a short proof of Theorem \ref{theoperturb} which combines the main idea in \cite{BD} with the results in \cite{ASS} on existence and properties of solutions of fully nonlinear equations in cones.

The main idea in \cite{BD} is essentially to reproduce an \emph{approximate} corner lemma. We will not be able to find a contradiction by analyzing the second derivatives of the solution $w_{s^*}$ at $Q$. Instead, we need to contradict a Taylor expansion of order $2+\alpha$ with a non degeneracy result of order strictly less than $2+\alpha$, which holds if $\Lambda/\lambda$ is sufficiently close to one. We are going to show how the contradiction argument can be carried out with the help of the following results from \cite{ASS}, which provide the required non-degeneracy results for domains with corners.

 The first proposition we need is part of Theorems 1.1 and 1.2 in \cite{ASS}.
\begin{prop}\label{fundyex} Let $\sigma\subset S^{n-1}$ be  open and smooth, and $\C$ be the projected cone,
\[ \C = \R_+^* \cdot \sigma = \{tx : t>0 \text{ and } x \in \sigma\} = \{x\in \rn\setminus\{0\} \,:\, |x|^{-1}x \in \sigma\} \]

 There exist a number $\beta >0$ and a $\beta$-homogeneous function $\Psi$ such that
$$
\Psi\in C(\overline{\C}),\qquad \mathcal{M}_{\lambda, \Lambda}^-(D^2\Psi) = 0\quad \mbox{ and }\quad \Psi >0 \quad \mbox{ in }\quad \C, \qquad \Psi = 0\quad \mbox{ on }\quad \partial \C.
$$
Any other solution of this problem is a multiple of $\Psi$.
\end{prop}

\begin{remark}
Studying the proof of this result in \cite{ASS} it is possible to see that it extends, with almost the same proof, to the case when $\sigma$ is only Lipschitz, such as the intersection of a quarter space with $S^{n-1}$. This fact could simplify even further the proof below, but we will not use it, for the readers' convenience. Note also that the number $\beta$ in Proposition \ref{fundyex} is equal to $-\alpha^-$ in the notations of \cite{ASS}, and
\begin{multline} \label{defbeta}
\beta : = \sup\left\{ \tilde{\beta} > 0 :  \ \mbox{there exists a } \tilde{\beta}- \mbox{homogeneous supersolution } \Phi\in \ C(\overline{\C}) \right. \\
\left. \mbox{of } \ \mathcal{M}_{\lambda, \Lambda}^-(D^2\Phi)  \leq 0 \ \mbox{and} \ u>0 \ \mbox{ in } \C \right\}
\end{multline}
\end{remark}

The following easy lemma says that the homogeneity of the function $\Psi$ from Proposition \ref{fundyex} is close to two when $\C$ is close to a quarter-space and the fully nonlinear operator is close to the Laplacian.

We denote with $\Pi$ the quarter-space (intersection of two half-spaces)
 $$
 \Pi = \{ x=(x_1,\ldots, x_n)\in \rn\::\: x_1>0, \;x_n>0\}.
 $$
Set $\pi = \Pi\cap S^{n-1}$, that is, $\Pi = \R_+^* \cdot \pi$.

\begin{lemma} \label{lemma-almostcorner} Let $\sigma_m$ be an increasing sequence of smooth subdomains of $\pi$ such that $\sigma_m\to\pi$ as $m\to \infty$. Let $\Psi_m$ be the $\beta_m$-homogeneous function given by Proposition~\ref{fundyex}, applied to the operator $\mathcal{M}_{\lambda, \lambda(1+1/m)}$ in $\C_m=\R_+^*\cdot \sigma_m$. Then $\beta_m\to 2$ as $m\to\infty$.
\end{lemma}

\begin{proof} Note that $\beta_m$ is nonincreasing, by \re{defbeta} and the definition of $\mathcal{M}_{\lambda, \Lambda}^-$.
If we normalize $\Psi_m$ so that  $\Psi_m(x_0)=1$ for a fixed point $x_0\in\Pi$, we can use the Harnack inequality and the elliptic H\"older estimates (the constants in these estimates depend only on the uniform exterior cone condition) to
 \beeq\label{local1}
 \mathcal{M}_{\lambda, \lambda(1+1/m)}(D^2\Psi_m)=0,
 \eeq
 and deduce that $\|\Psi_m\|_{C^\alpha(K)}\le C(K)$ for each compact subset $K$ of $\pi$ and all large $m$, where $C(K)$ is a constant independent of $m$.
  Hence,  using the stability properties of viscosity solutions with respect to uniform convergence, we can pass to the limit in \re{local1} and conclude that $\Psi_m$ converges locally uniformly as $m\to \infty$ to the unique (up to a multiplication by a constant) positive harmonic function in $\Pi$ which vanishes on $\partial\Pi$. Of course, this function is $x_1x_n$ and its homogeneity is two.
\end{proof}

The next proposition is essentially  Theorem 1.4 in \cite{ASS}.
\begin{prop}\label{hopf} Let $\sigma\subset S^{n-1}$ be open and smooth, $\C = \R_+^* \cdot \sigma$, and  $\C_0:=\C\cap B_{\epsilon_0}$ for some $\epsilon_0>0$. Assume $\beta\ge 1$, where $\beta$ is the number defined above for the cone $\C$.  Let $\Sigma_0$ be a domain such that $0\in \partial \Sigma_0$ and $\Sigma_0$ is $C^2$-diffeomorphic to $\C_0$. If  $w\in C(\bar \Sigma_0)$ is nonnegative and satisfies
\beeq\label{ineqa}
\mathcal{M}_{\lambda, \Lambda}^-(D^2 w) - k|D w| -lw\le 0\quad \mbox{ in }\;\Sigma_0
\eeq
in the viscosity sense, then
 either $w\equiv0$ in $\Sigma_0$ or
\beeq\label{ineq2}
\liminf_{t\searrow 0} \frac{w(te)}{t^{\beta}}>0,
\eeq
for any  direction $e\in S^{n-1}$ which enters $\Sigma_0$.
\end{prop}

\begin{remark}
Note that if $\Sigma_0$ is $C^2$-smooth around the origin then the corresponding cone $\C$ is a half space, say $\C=\{x_n=0\}$, so $\Psi= x_n$ and $\beta=1$, independently of $\lambda,\Lambda$. Then Proposition \ref{hopf} becomes the statement of the usual and well-known Hopf lemma.
\end{remark}

Proposition \ref{hopf} can be proved with practically the same proof as Theorem 1.4 in \cite{ASS}, since the first and zero order terms in \re{ineqa} ``scale out" when we zoom into the origin. However, since the proof in \cite{ASS} is done in conjunction with other results in that paper, and is thus not simple to follow, for the reader's convenience we present a simpler and self-contained proof of Proposition \ref{hopf}, in the Appendix below. We note the assumption $\beta\ge1$ in Proposition \ref{hopf} can be removed, see \cite{ASS}.

\begin{proof}[Proof of  Theorem \ref{theoperturb}]
We follow the moving planes method as explained in section~\ref{sect-movingplanes}. We need to show that the alternative (ii) (the corner situation) cannot happen. Without loss of generality we assume that $\lst=0$, $Q$ is the origin in $\rn$,  and we will write $\bar x$ instead of $x^{\lst}$; see the notation of section \ref{sect-movingplanes}.

Recall that in the corner situation for the moving plane method we have a set $\Sigma$ which around the origin is $C^2$-diffeomorphic to a neighborhood of the origin in $\Pi$. Obviously we can assume $\Pi$ is determined by $T$ and the tangent plane to $\partial \Omega$ at the origin.

Recall also that we have a function $w:\Sigma \to \R_+$ (we drop the subscript $\lst=0$) which is nonnegative in a domain $\Sigma$ with a right-angle corner point at the origin, and vanishes at this point together with its derivatives up to order two (the latter is to be understood in the sense of Proposition \ref{prop-expansion-on-the-boundary}, in case $u$ is not $C^2$ close to the boundary of $\Omega$). In addition, \re{equa2} holds thanks to Proposition \ref{prop-expansion-on-the-boundary}, that is
\beeq\label{ineq1}
w(x)\le C|x|^{2+\alpha}
\eeq
for any $x \in \Sigma$.
We will now contradict this inequality with the help of Lemma \ref{lemma-almostcorner}.

The set $\Sigma$ is $C^2$-diffeomorphic to  a straight corner around the origin. Thus, we can find a sequence of smooth cones $\C_m = \R_+ \cdot \sigma_m$, such that $\sigma_m\to\pi$ from inside, and there is $r_m>0$ such that $\C_m \cap B_r \subset \Sigma \cap B_r$ for all  $r\in (0,r_m)$.

From Lemma \ref{lemma-almostcorner}, for each $\C_m$ we have a function $\Psi_m$ which is  homogeneous of degree $\beta_m$, and $\beta_m \to 2$ as $m \to \infty$.
Recall that for some $k,l\ge 0$
$$
\mathcal{M}_{\lambda, \Lambda}^-(D^2w) -k|D w| -lw\le 0\qquad \mbox{in }\; \Sigma.
$$
Applying Proposition \ref{hopf} in each $\C_m \cap B_r$, we obtain that $w(te) \geq c_mt^{\beta_m}$ for each direction $e$ that enters $\C_m$. But this is a contradiction with \eqref{ineq1} when $m$ is fixed so large that $\beta_m < 2+\alpha$, and $t$ is very small.
\end{proof}

\section{Proof of Theorem \ref{theoc1}}
\label{sect-proofC1}

In this section we prove the first of our main theorems. The argument from the preceding section will play an important role here; we will actually reduce the proof of Theorem \ref{theoc1} to that of Theorem \ref{theoperturb}.

\begin{proof}[Proof of Theorem \ref{theoc1}]
Again, we need to find a contradiction in the case (ii), the corner situation, in the moving planes method explained in section \ref{sect-movingplanes}. Without loss of generality we assume that $\lst=0$, $Q$ is the origin in $\rn$,  and we write $\bar x$ instead of $x^{\lst}$.

For every symmetric matrix $M=(m_{ij})\in \S_n$ we denote with $\ov{M}$ the matrix with entries $\ep_{ij}m_{ij}$ where $\ep_{11}= 1$, $\ep_{ij}= 1$ if $i,j\ge2$, and $\ep_{1j}= -1$ if $j\not=1$. Observe that spec$(M) =$ spec$(\ov{M})$, so $F(M,p)  = F(\ov{M},p)$ for each $M\in \S_n$, $p\in \R_+$.

We know that $u\in C^{2,\alpha}(\Omega_\delta)$, for some $\delta$-neighbourhood of the boundary $\partial \Omega$ in $\Omega$ (this is by hypothesis if $(H2)^\prime$ is assumed, and follows from Proposition \ref{bdaryC2alpha} in case we assume (H2)). We set $\Sigma_0= \Sigma\cap B_{\delta/2}$.

 Next, note  that we have
$$
D^2v(x) = \ov{D^2u(\bar x)}\quad\mbox{ for each }\; x\in\ov{\Sigma_0},
$$
and in particular
$$
D^2v = \ov{D^2u} \quad\mbox{ on the hyperplane } T_0\cap\partial\Sigma_0 .
$$

Since $F$ is continuously differentiable and $f$ is locally Lipschitz, the function $w=v-u$ is a solution of a linear equation
\beeq\label{fiu}
\mathrm{tr}(A(x)D^2w) + \lsc b(x),D w\rsc + c(x) w = 0\quad\mbox{ in }\; {\Sigma_0}
\eeq
where $b,c\in L^\infty(\Sigma_0)$ and the matrix $A=(a_{ij})$ has continuous entries in $\ov{\Sigma_0}$,
\begeqaet
2a_{ij} (x) &=& 2 \int_0^1 \frac{\partial F}{\partial m_{ij}} (tD^2u(x) +  (1-t) D^2v(x))\,dt\\
&=&   \int_0^1 \frac{\partial F}{\partial m_{ij}} (tD^2u(x) +  (1-t) D^2v(x))+ \frac{\partial F}{\partial m_{ij}} (tD^2v(x) +  (1-t) {D^2u(x)}).
\eneqaet

Since $F(M)  = F(\ov{M})$ implies $\displaystyle\frac{\partial F}{\partial m_{1j}}(M) + \frac{\partial F}{\partial m_{1j}}(\ov{M})=0$ for  $j>1$, we get
\beeq\label{tre}
a_{1j} = 0 \quad\mbox{ on the hyperplane } T_0\cap\partial\Sigma_0 , \mbox{ for each } j>1.
\eeq

\begin{remark}
If $F(M)$ is in addition assumed to be in $C^{1,1}(\S_n)$ and $u\in C^2(\ov{\Omega})$ then it follows from \re{tre}  that Serrin's corner lemma in its general version (Lemma 2 in \cite{Se}) applies to \re{fiu} in the whole of $\Sigma$, which ends the proof of Theorem \ref{theoc1}. This fact was already observed by Reichel in his proof of the symmetry for the Monge-Amp\`ere operator in \cite{Re2}.
\end{remark}

 As a matter of fact, we only need  that $a_{1j}(0)=0$ for $j>1$, and the continuity of $a_{1j}$ in a neighbourhood of the origin. Then we can make a change of coordinates which consists of a rotation in the tangent plane to $\partial \Omega$ at $0$ diagonalizing the minor $(a_{ij}(0))_{i,j=2}^n$, and of a stretch of the coordinate vectors, so that in the new coordinates  the modified function $\tilde{w}$
satisfies
\beeq\label{fiu1}
\mathrm{tr}(\tilde{A}(x)D^2\tilde{w}) + \lsc \tilde{b}(x),D \tilde{w}\rsc+ \tilde{c}(x) \tilde{w} = 0\quad \mbox{in }\; \widetilde{\Sigma_0}
\eeq
where $\tilde{b}$, $\tilde{c}$ are bounded, $\tilde{A}$ is uniformly elliptic and continuous in $\widetilde{\Sigma_0}$, $\tilde{A}(0)= I$  and $ \widetilde{\Sigma_0}$ { still has  a right-angle corner at the origin}. Since $\tilde{A}(x)$ is continuous at the origin and $\tilde{A}(0)= I$, \re{fiu1} easily implies that for each $\varepsilon>0$ there exists $\delta>0$ such that
$$
\mathcal{M}_{\lambda, \Lambda}^-(D^2\tilde{w})-k|D \tilde{w}| -l\tilde{w}\le 0\quad \mbox{in }\; \tilde{\Sigma_0}\cap B_\delta
$$
 and
 $$
 1-\varepsilon<\lambda\le \Lambda<1+\varepsilon.
  $$
Therefore if $\ep$ is sufficiently small,  the same argument as in the proof of Theorem \ref{theoperturb} yields $\tilde{w}\equiv 0 $ in $\tilde{\Sigma_0}$. This finishes the proof of Theorem \ref{theoc1}.
\end{proof}
\medskip

\begin{remark}\label{impremark}
An important observation, which will be crucial for the proof of Theorems \ref{theon2} and \ref{theoconv} is that the above argument only requires that $F$ be $C^1$ in a neighbourhood of $(D^2u(0), |D u(0)|)$. This will be explained in more detail below.
\end{remark}

\section{The Pucci Equations}
\label{sect-pucci}

In this section we prove Theorem \ref{theon2} and Theorem \ref{theoconv}. The main idea is to reduce the proof to the argument in the preceding section, by showing that, for sufficiently many directions, if the moving plane stops at a point of orthogonality, then the Hessian of $u$ is non-singular at that point, which permits to us to use Remark \ref{impremark} and conclude.

\subsection{Towards the non-degeneracy of $D^2u$ at $Q$.}
In this subsection we prove two auxiliary results which, combined with Proposition \ref{lem1}, will help us to deduce that the matrix $D^2u$ is invertible at points of orthogonality reached by the moving planes, as explained in Section \ref{sect-movingplanes}.

The following measure-theoretic lemma contains a crucial observation.

\begin{lemma}\label{meas} Let $K(x)=\prod_{i=1}^{n-1} \kappa_i(x)$ denote the Gauss curvature of the boundary $\partial\Omega$ and $\Theta\subset\partial\Omega$ denote the set of points where at least one of the principal curvatures of $\partial \Omega $ vanishes:
$$
\Theta :=\{ x\in \partial\Omega \::\: K(x)=0\}.
$$
Then  $\nu(\Theta)$, the set of all normals at points of $\Theta$, is a negligible subset of the sphere $S^{n-1}$ (in the ($n-1$)-dimensional measure on $S^{n-1}$).
\end{lemma}

\begin{proof}
Consider the unit normal map $\nu : \partial \Omega \to S^{n-1}$. It is well known that the tangent space to $\partial \Omega$ at a point $x$ coincides with the tangent space of $S^{n-1}$ at $\nu(x)$. By definition, the second fundamental form of $\partial\Omega$ at $x$ is $D\nu: T_{\partial \Omega}(x) \to T_{\partial \Omega}(x)$ and the Gaussian curvature equals $\det D\nu$.

For any set $A \subset \partial \Omega$, we have the following area formula for the measure of the image $\nu(A)$
\[ |\nu(A)| = \int_A \det D\nu(x) \ \mathrm{d}x = \int_A K(x) \ \mathrm{d}x.\]
Therefore, if the Gaussian curvature $K(x)$ vanishes in the whole of $A$, then we have $|\nu(A)|=0$.
In particular, since $K\equiv 0$ in $\Theta$ by definition, we have $|\nu(\Theta)|=0$.
\end{proof}

When a plane  moving in a direction $e$ stops in a corner situation at a point $Q$, we know that $e$ must be a tangent vector to $\partial \Omega$ at $Q$. The reason why the two dimensional case is special is because for each unit normal vector correspond only two unit tangent vectors (one opposite to each other). Thus, the previous result will help us to obtain sufficiently many directions $e$ for which the method succeeds and $\Omega$ is symmetric.
\medskip

Next we record a result pertaining to the non-degeneracy of the solution of \re{princ} in  directions normal to the boundary of $\Omega$, under the hypothesis that the function $f$ in \re{princ} is nonincreasing. This hypothesis is more restrictive than $f(0)\ge 0$, since if $f$ is nonincreasing and $f(0)< 0$ then $f$ is negative and the existence of a solution of \re{princ} contradicts the maximum principle and (H2).

The next lemma says the second normal derivative of $u$ does not vanish on the whole boundary of $\Omega$, in case $f$ is nonincreasing. So, recalling Proposition \ref{lem1} and Remark~\ref{remafterprop25},  in that case the non-degeneracy of the Hessian of $u$ at any boundary point $Q$ is equivalent to  the Gauss curvature of $\partial \Omega$ being different from zero at $Q$.

\begin{lemma}\label{unn} Assume  (H2), and that  $f$ is nonincreasing. Then for each $u$  solution of \re{princ} and each $x_0\in \partial \Omega$ we have
$$
u_{\nu\nu}(x_0)<0,
$$
where $\nu=\nu(x_0)$ is the interior  normal to $\partial \Omega$ at $x_0$.

If $u$ is not $C^2$ in a neighbourhood of $x_0$, the second derivative $u_{\nu \nu}$ is to be understood in the sense of Proposition \ref{prop-expansion-on-the-boundary}.
\end{lemma}

\begin{proof}
Fix $x_0\in\partial \Omega$ and set $\nu_0=\nu(x_0)$. A simple, yet basic observation is that the function
$$
v(x):=\frac{\partial u}{\partial \nu_0}(x)=\lsc \nu_0,Du(x)\rsc, \quad x\in \ov{\Omega},
$$
is a viscosity solution of the inequalities
\beeq\label{ineq3}
\mathcal{M}_{\lambda, \Lambda}^-(D^2{v})-k|D v| +l(x){v}\le 0, \quad \mathcal{M}_{\lambda, \Lambda}^+(D^2{v})+k|D v| +l(x){v}\ge 0 \quad \mbox{in }\; \Omega,
\eeq
where $l(x)$ is bounded and nonpositive in $\Omega$, since $f$ is locally Lipschitz and nonincreasing.  Here we will actually use only the second inequality in \re{ineq3}.

In order to see that \re{ineq3} hold we may observe \re{princ} is satisfied by both functions $u(x)$ and $u(x+h\nu_0)$, subtract the two equations and use (H2), then divide by $h$ and note that the sequence $v_h(x):=h^{-1}(u(x+h\nu_0)-u(x))$ converges as $h\to0$ locally uniformly to $v$ on $\Omega$. Hence the stability properties of viscosity solutions with respect to uniform convergence yield \re{ineq3}.

Furthermore, for each $x\in\partial\Omega$, the Neumann boundary condition in \re{princ} implies
$$
v(x)=\lsc \nu_0,c_0\nu(x)\rsc \, \le c_0 |\nu_0||\nu(x)|=c_0.
$$
Hence we can apply the maximum principle to the function $w=c_0-v$ which satisfies
 $$
 -\mathcal{M}_{\lambda, \Lambda}^-(D^2{w})+k|D w| -c_0l(x){w}\ge -c_0l(x)\ge 0\quad \mbox{in }\; \Omega,
$$
 and deduce that $v(x)\le c_0$ in $\Omega$. Then, since
$$
v(x_0)=c,
$$
 we can apply  the Hopf lemma  and infer that
$$
\limsup_{t\searrow0}\frac{v(x_0+t\nu_0)-v(x_0)}{t}<0,
$$
which is the conclusion of Lemma \ref{unn}.\end{proof}

\begin{remark} We will not use Lemma \ref{unn} in the sequel but nevertheless include it here since it both simplifies the proof below and may turn out to be useful in the future, for instance in attempting to extend Theorem \ref{theon2} to higher dimensions.
\end{remark}

\subsection{Proofs of Theorems \ref{theon2} and \ref{theoconv}}.
We start by observing once more that in the proof of Theorem \ref{theoc1} we only used that $F$ is $C^1$ in a neighborhood of $(D^2u(Q),|D u(Q)|)$, since in that case Proposition \ref{bdaryC2alpha-local} implies that the solution $u$ belongs to the class $C^{2,\alpha}$ in a neghbourhood of $Q$ in $\ov{\Omega}$, and the rest of the proof of Theorem \ref{theoc1} is unchanged. Recall that we call $Q$ the point of orthogonality on $\partial \Omega$ where the  plane moving in the direction $e$ stops, in the sense of section \ref{sect-movingplanes}.

More precisely, going over the proof of Theorem \ref{theoc1}, we see that there we proved the following result.

\begin{prop} \label{proplocal} Assume that (H1) and (H2) hold. Assume that when applying the moving planes method in a direction $e\in S^{n-1}$, the process stops at a corner situation, with a right angle at a point $Q$. If  the operator $F$ is continuously differentiable in a neighborhood of $(D^2u(Q),|D u(Q)|)$ in $\S_n\times(0,\infty)$ then $\Omega$ is symmetric in the direction~$e$.
\end{prop}

We see that the hypothesis (H3) is taylor-made for applying this proposition. In particular, this hypothesis is well suited to any extremal operator in the form
\beeq\label{extre}
\mathcal{M}_{\lambda, \Lambda}^\pm(D^2u)\pm k|D u|.
\eeq

Indeeed, the non-degeneracy of the gradient term in this operator in a neighborhood of the boundary is guaranteed by $u\in C^{1,\alpha}(\Omega)$ (recall Proposition \ref{c1alpha}) and  the Neumann hypothesis $|D u|= c_0\not = 0 $ on $\partial \Omega$. Hence the success of the moving planes method in the direction $e\in S^{n-1}$ for the extremal operator in \re{extre} is guaranteed by Proposition~\ref{proplocal} provided the Pucci operator is $C^1$ in a neighborhood of $D^2u(Q)$. By the definition of Pucci operators this is equivalent to det$D^2u(Q)\not = 0$. Here, as before, if $u$ is not a priori assumed to be in $C^2$ close to the boundary, then $D^2u(Q)$ is to be understood in the sense of Proposition \ref{prop-expansion-on-the-boundary}.

Finally, recall that the boundary of $\Omega$ is a level set both of $u$ and $|Du|$, and hence Lemma \ref{lem1} applies and shows that det$D^2u(Q)\not = 0$ is equivalent to $K(Q)\not=0$ and $u_{\nu\nu}(Q)\not = 0$, where $K(Q)=\prod_{k=1}^{n-1}\kappa_i(Q) $ is the Gauss curvature of $\partial \Omega$ at $Q$.  It now remains to check that these two conditions are met under the assumptions of Theorems~\ref{theon2} and \ref{theoconv}.

Recall that $f(0)\ge 0$. If we assume in addition that all principal curvatures of $\partial \Omega$ at $Q$ are nonnegative and  $u_{\nu\nu}(Q)=0$,   we obtain
\begeqaet
\lambda\sum_{i=1}^{n-1} \kappa_i(Q)-kc_0&=&\mathcal{M}_{\lambda, \Lambda}^-(D^2{u}(Q))-k|Du(Q)|\\
&\le& F(D^2{u}(Q), |Du(Q)|)\le -f(0)\le 0,
\eneqaet
Assume now that the mean curvature of $\partial \Omega$ at $Q$ is strictly positive. As we explained at the end of section \ref{sect-movingplanes}, by dividing the solution $u$ by a large constant we can assume that $c_0$ is as small as we like, without modifying $\lambda, \Lambda$, and $k$. The modified function satisfies an equation to which the moving planes method applies and stops at the same position as for $u$. Hence, fixing $c_0 \le (\lambda/2k)\sum_{i=1}^{n-1} \kappa_i(Q)$, the above inequality is impossible. In other words, what we just proved is that $u_{\nu\nu}$ does not vanish at points of the boundary at which all principal curvatures are nonnegative and one of them is positive.

At this point the proof of Theorem \ref{theoconv} is finished, since in this theorem we
{\it assume} that all the principal curvatures of  $\partial \Omega$ are strictly positive at all points of $\partial \Omega$.

In the end, let us conclude the proof of Theorem \ref{theon2}. We first observe that, with the notations of the moving planes method we explained in section \ref{sect-movingplanes}, we always have $k(x)\ge 0$ for all $x\in \partial \Omega\cap \partial \Sigma_{s^*}$, where $k(x)$ denotes the curvature of $\partial\Omega$ at the point~$x$. In other words, a moving plane never reaches points of negative curvature (since it would reach a point of orthogonality before reaching such points). So the moving planes method succeeds in a direction $e\in S^1$ provided the point of orthogonality on the boundary which is reached by the moving plane has non-zero curvature.

However, by Lemma \ref{meas}, we know that there is a set of directions with {\it full measure} on $S^{1}$ such that if a plane moving in one of these directions stops at a point of orthogonality $Q$, then the curvature of the boundary at $Q$  is not zero. Here we use the assumption $n=2$ which implies that to each point on the boundary there correspond only two directions which are orthogonal to the  normal at this point.

Therefore, we obtain that for a set of directions $e$ in $S^{n-1}$ with full measure, the set $\Omega$ is symmetric with respect to $e$. In particular this is a dense set of directions in $S^{n-1}$. By density, we can extend the symmetry to all directions in $S^{n-1}$ which implies that $\Omega$ is a ball.

Theorems \ref{theon2} and \ref{theoconv} are proved. \hfill $\Box$
\medskip

\begin{remark}
For dimensions higher than two, we see that the moving plane method can only fail if a plane moves in a direction $e$ which corresponds to a tangent vector at a point $Q$ with Gauss curvature zero. From
Lemma \ref{meas} we know that the normal vectors at these points form a set of measure zero in $S^{n-1}$. However, the vector $e$ belongs to the set of orthogonal vectors to $\nu(\Theta)$, which may have positive measure in $S^{n-1}$. Because of this, currently we do not know how to prove an analog to Theorem~\ref{theon2} in higher dimension and we leave it as an open problem.
\end{remark}

\section{Appendix}

\begin{proof}[Proof of Proposition \ref{hopf}]

It is a standard fact that if $\Phi : \Sigma _0 \to \C_0$ is a $C^2$-diffeomor\-phism, and $\tilde{u}(x) = u(\Phi^{-1}(x))$ then $\tilde{u}$ is a solution in $\C_0$ of the same inequality as $u$ in $\Sigma_0$, with possibly modified constants $\lambda, \Lambda, k, l$ depending only on the $C^2$-norm of $\Phi$. So, without restricting the generality we may assume that $\Sigma _0 =\C_0$.

Let $r_0=\epsilon_0/8$.  For all $r<r_0$, let $\psi_r$ be the unique solution of the problem
\begin{equation}\label{isol}
\left\{ \begin{aligned}
&\mathcal{M}_{\lambda, \Lambda}^-(D^2\psi_r)-kr|D\psi_r| -lr^2 \psi_r = 0 & \mbox{in} & \ \C\cap B_4, \\
& \psi_r=\Psi & \mbox{on} & \ \partial(\C\cap B_4),
\end{aligned} \right.
\end{equation}
where $\Psi$ is the function given by Proposition \ref{fundyex}. Note the elliptic operator in \re{isol} is proper.

By the ABP inequality (see for instance Theorem 1.7 in \cite{QS}), there exists a constant $C$ independent of $r\in (0, r_0)$ such that
$$
\sup_{\C\cap B_4} \psi_r \le C\sup_{\C\cap B_4} \Psi = C.
$$
($C$ will change from line to line but stays independent of $r$).

By the boundary Lipschitz estimates (these estimates are particularly simple to prove for our cones, since they satisfy an uniform exterior sphere condition; see for instance Proposition 4.9 in \cite{QS}) and the fact that the function $\Psi$ has bounded gradient (because $\beta\ge1$ and $\Psi$ is $\beta$-homogeneous) we infer that each directional derivative of $\psi_r$ satisfies the inequalities \re{ineq3} in $\C\cap B_4$ and is bounded at the  boundary of $\C\cap B_4$. Applying the ABP inequality to \re{ineq3} we obtain
$$
\|D\psi_r\|_{L^\infty(\C\cap B_4)}\leq C.
$$

Since
\begin{equation}\label{ineq5}
\left\{ \begin{aligned}
&\mathcal{M}_{\lambda, \Lambda}^+(D^2(\psi_r-\Psi)) \ge -kr|D\psi_r|- lr^2 \psi_r\ge -Cr,\\ & \mathcal{M}_{\lambda, \Lambda}^-(D^2(\psi_r-\Psi)) \le kr|D\psi_r|+lr^2 \psi_r\le Cr
\end{aligned}
\right.
\end{equation}
in $\C\cap B_4$ and $\psi_r-\Psi =0$ on $\partial(\C\cap B_4)$, by the ABP inequality again
$$
\|\psi_r-\Psi\|_{L^\infty(\C\cap B_4)}\le Cr.
$$
By the boundary Lipschitz estimates  applied to \re{ineq5} we  get
$$
|\psi_r(x)-\Psi(x)|\le C r|x-x_0|
$$
for each $x\in \C\cap B_4$ and $x_0\in \partial(\C\cap B_4)$.

 By the Hopf lemma,
\begin{equation*}
\Psi(x) \geq \frac{1}{C} \mbox{dist}(x, \partial \C) \quad \mbox{in } \C\cap(B_2 \setminus B_{1}),
\end{equation*}
and hence
\beeq\label{iniq}
\left|\frac{\psi_r}{\Psi}-1\right| \le Cr \quad \mbox{in } \C\cap(B_2 \setminus B_{1}).
\eeq

Next, set $w_r(x) = r^{\beta} w(rx)$ and
$$
q(r) := \inf_{\C\cap(B_{2r} \setminus B_{r}) }\frac{w}{\Psi} = \inf_{\C\cap(B_{2} \setminus B_{1}) }\frac{w_r}{\Psi}
$$
 (recall that $\Psi(rx) = r^{\beta} \Psi(x)$). We have, by \re{ineqa} and $r<\epsilon_0/4$,
 \beeq\label{ineq6}
 \mathcal{M}_{\lambda, \Lambda}^-(D^2w_r)-k|Dw_r| -lr^2 w_r \le 0 \quad\mbox{in }\; \C\cap B_4
 \eeq
  and $w_r\ge q(2r) \psi_r $ on $\partial(\C\cap B_4)$. Hence we can apply the maximum principle  to \re{isol} and \re{ineq6}, and deduce that $w_r\ge q(2r) \psi_r $ in $\C\cap B_4$. Then  \re{iniq} yields
 $$
 w_r \ge q(2r) (1-Cr)\Psi $$
 in $\C\cap (B_2\setminus B_1)$, which in particular implies
 \beeq\label{kwas}
 q(r)\ge q(2r) (1-Cr)
 \eeq
 for all small $r\in (0,1)$. An iteration argument, which we give for completeness, then implies that $q(r)$ is bounded away from zero as $r\in (0,1)$. Specifically, \re{kwas} implies that for all $k,l\in \mathbb{N}$, $k<l$,
 $$
 q\left(\frac{1}{2^{l}}\right) \ge \displaystyle\mathop{\prod}_{s=k}^l \left(1-C\left(\frac{1}{2^s}\right)\right) q\left(\frac{1}{2^{k}}\right)
 $$
 So, with $k_1$ fixed so that $C/2^{k_1-1}\le 1$, by using the inequality
 $$
 \ln(1-y)\ge -2y$$
 valid for $y\in (0,1/2)$ we obtain
 $$
 q(2^{-l})\ge e^{-Cr_1}q(r_1)
 $$
 for all $l>k_0$, where $r_1= 2^{-k_1}$.  Observe that for each $r\in (0,1)$ there exists $k\in \mathbb{N}$ such that $r\in [2^{-k},2^{-(k+1)})$ and that then the definition of $q$ implies
 $$
 q(r)\ge \min\left\{ q(2^{-k}),q(2^{-(k+1)})\right\}.
 $$
 Hence
 $$
 q(r)\ge e^{-Cr_1}q(r_1)>0
 $$
for $r\in (0,r_1)$.
\end{proof}

\section*{Acknowledgements}

Luis Silvestre was partially supported by NSF grants DMS-1065979 and DMS-1001629.

\bibliographystyle{plain}
\bibliography{overdet}

\begin{thebibliography}{10}

\bibitem{AM}
Virginia Agostiniani and Rolando Magnanini.
\newblock Symmetries in an overdetermined problem for the {G}reen's function.
\newblock {\em Discrete Contin. Dyn. Syst. Ser. S}, 4(4):791--800, 2011.

\bibitem{ASS}
Scott~N. Armstrong, Boyan Sirakov, and Charles~K. Smart.
\newblock Singular solutions of fully nonlinear elliptic equations and
  applications.
\newblock {\em Arch. Ration. Mech. Anal.}, 205(2):345--394, 2012.

\bibitem{BN}
H.~Berestycki and L.~Nirenberg.
\newblock On the method of moving planes and the sliding method.
\newblock {\em Bol. Soc. Brasil. Mat. (N.S.)}, 22(1):1--37, 1991.

\bibitem{BD}
I.~Birindelli and F.~Demengel.
\newblock Overdetermined problems for some fully non linear operators.
\newblock {\em Comm. Part. Diff. Eq.}, 38(4):608--628, 2013.

\bibitem{BNST}
B.~Brandolini, C.~Nitsch, P.~Salani, and C.~Trombetti.
\newblock Serrin-type overdetermined problems: an alternative proof.
\newblock {\em Archive for Rational Mechanics and Analysis}, 190(2):267--280,
  2008.

\bibitem{BH}
F.~Brock and A.~Henrot.
\newblock A symmetry result for an overdetermined elliptic problem using
  continuous rearrangement and domain derivative.
\newblock {\em Rend. Circ. Mat. Palermo (2)}, 51(3):375--390, 2002.

\bibitem{BK}
G.~Buttazzo and B.~Kawohl.
\newblock Overdetermined boundary value problems for the {$\infty$}-laplacian.
\newblock {\em International Mathematics Research Notices}, 2011(2):237--247,
  2011.

\bibitem{CH}
Mourad Choulli and Antoine Henrot.
\newblock Use of the domain derivative to prove symmetry results in partial
  differential equations.
\newblock {\em Math. Nachr.}, 192:91--103, 1998.

\bibitem{CS}
Andrea Cianchi and Paolo Salani.
\newblock Overdetermined anisotropic elliptic problems.
\newblock {\em Math. Ann.}, 345(4):859--881, 2009.

\bibitem{dLS}
Francesca Da~Lio and Boyan Sirakov.
\newblock Symmetry results for viscosity solutions of fully nonlinear uniformly
  elliptic equations.
\newblock {\em J. Eur. Math. Soc. (JEMS)}, 9(2):317--330, 2007.

\bibitem{En}
Cristian Enache.
\newblock Maximum principles and symmetry results for a class of fully
  nonlinear elliptic {PDE}s.
\newblock {\em NoDEA Nonlinear Differential Equations Appl.}, 17(5):591--600,
  2010.

\bibitem{ES}
Cristian Enache and Shigeru Sakaguchi.
\newblock Some fully nonlinear elliptic boundary value problems with
  ellipsoidal free boundaries.
\newblock {\em Math. Nachr.}, 284(14-15):1872--1879, 2011.

\bibitem{FK}
A.~Farina and B.~Kawohl.
\newblock Remarks on an overdetermined boundary value problem.
\newblock {\em Calc. Var. Partial Differential Equations}, 31(3):351--357,
  2008.

\bibitem{FV1}
Alberto Farina and Enrico Valdinoci.
\newblock Flattening results for elliptic {PDE}s in unbounded domains with
  applications to overdetermined problems.
\newblock {\em Arch. Ration. Mech. Anal.}, 195(3):1025--1058, 2010.

\bibitem{FV2}
Alberto Farina and Enrico Valdinoci.
\newblock Overdetermined problems in unbounded domains with {L}ipschitz
  singularities.
\newblock {\em Rev. Mat. Iberoam.}, 26(3):965--974, 2010.

\bibitem{FG}
Ilaria Fragal{\`a} and Filippo Gazzola.
\newblock Partially overdetermined elliptic boundary value problems.
\newblock {\em J. Differential Equations}, 245(5):1299--1322, 2008.

\bibitem{FGK}
Ilaria Fragal{\`a}, Filippo Gazzola, and Bernd Kawohl.
\newblock Overdetermined problems with possibly degenerate ellipticity, a
  geometric approach.
\newblock {\em Math. Z.}, 254(1):117--132, 2006.

\bibitem{GL}
N.~Garofalo and J.L. Lewis.
\newblock A symmetry result related to some overdetermined boundary value
  problems.
\newblock {\em American Journal of Mathematics}, pages 9--33, 1989.

\bibitem{HHP}
Laurent Hauswirth, Fr{\'e}d{\'e}ric H{\'e}lein, and Frank Pacard.
\newblock On an overdetermined elliptic problem.
\newblock {\em Pacific J. Math.}, 250(2):319--334, 2011.

\bibitem{krylov1983}
N.~V. Krylov.
\newblock Boundedly inhomogeneous elliptic and parabolic equations in a domain.
\newblock {\em Izv. Akad. Nauk SSSR Ser. Mat.}, 47(1):75--108, 1983.
\newblock English translation in Math. USSR Izv. 22, 67-97 (1984).

\bibitem{Li}
Congming Li.
\newblock Monotonicity and symmetry of solutions of fully nonlinear elliptic
  equations on unbounded domains.
\newblock {\em Comm. Partial Differential Equations}, 16(4-5):585--615, 1991.

\bibitem{MS}
Rolando Magnanini and Shigeru Sakaguchi.
\newblock Matzoh ball soup: heat conductors with a stationary isothermic
  surface.
\newblock {\em Ann. of Math. (2)}, 156(3):931--946, 2002.

\bibitem{NV}
Nikolai Nadirashvili and Serge Vl{\u{a}}du{\c{t}}.
\newblock Singular solutions of {H}essian fully nonlinear elliptic equations.
\newblock {\em Adv. Math.}, 228(3):1718--1741, 2011.

\bibitem{Pr}
J.~Prajapat.
\newblock Serrin’s result for domains with a corner or cusp.
\newblock {\em Duke mathematical journal}, 91(1):29--31, 1998.

\bibitem{QS}
A.~Quaas and B.~Sirakov.
\newblock Principal eigenvalues and the dirichlet problem for fully nonlinear
  elliptic operators.
\newblock {\em Advances in Mathematics}, 218(1):105--135, 2008.

\bibitem{Ra}
A.~G. Ramm.
\newblock Symmetry problem.
\newblock {\em Proc. Amer. Math. Soc.}, 141(2):515--521, 2013.

\bibitem{Re2}
W.~Reichel.
\newblock Radial symmetry for an electrostatic, a capillarity and some fully
  nonlinear overdetermined problems on exterior domains.
\newblock {\em Z. Anal. Anwendungen}, 15(3):619--635, 1996.

\bibitem{Re1}
Wolfgang Reichel.
\newblock Radial symmetry for elliptic boundary-value problems on exterior
  domains.
\newblock {\em Arch. Rational Mech. Anal.}, 137(4):381--394, 1997.

\bibitem{Se}
J.~Serrin.
\newblock A symmetry problem in potential theory.
\newblock {\em Archive for Rational Mechanics and Analysis}, 43(4):304--318,
  1971.

\bibitem{SS}
L.~Silvestre and B.~Sirakov.
\newblock Boundary regularity for viscosity solutions of fully nonlinear
  elliptic equations.
\newblock {\em preprint, arXiv}, 2013.

\bibitem{Si}
Boyan Sirakov.
\newblock Symmetry for exterior elliptic problems and two conjectures in
  potential theory.
\newblock {\em Ann. Inst. H. Poincar\'e Anal. Non Lin\'eaire}, 18(2):135--156,
  2001.

\bibitem{We}
H.F. Weinberger.
\newblock Remark on the preceding paper of serrin.
\newblock {\em Archive for Rational Mechanics and Analysis}, 43(4):319--320,
  1971.

\end{thebibliography}
\end{document}